\documentclass{article}
\usepackage{amsmath,amsthm,amssymb,amscd}

\newcommand{\ga}{\alpha}
\newcommand{\gb}{\beta}
\renewcommand{\gg}{\gamma}

\newcommand{\gw}{\omega}

\newcommand{\gS}{\Sigma}
\newcommand{\gs}{\sigma}
\newcommand{\eps}{\varepsilon}

\newcommand{\coll}{\mathrm{Coll}}

\newcommand{\cantor}{2^\gw}
\newcommand{\baire}{\gw^\gw}

\newcommand{\supp}{\mathrm{supp}}

\newcommand{\dom}{\mathrm{dom}}

\newcommand{\rng}{\mathrm{rng}}

\newtheorem{theorem}{Theorem}[section]

\newtheorem{claim}[theorem]{Claim}
\newtheorem{corollary}[theorem]{Corollary}
\newtheorem{fact}[theorem]{Fact}
\newtheorem{proposition}[theorem]{Proposition}

\theoremstyle{definition}
\newtheorem{definition}[theorem]{Definition}
\newtheorem{example}[theorem]{Example}
\newtheorem{question}[theorem]{Question}
\title{Coloring closed Noetherian graphs\footnote{2000 AMS subject classification 03E35, 05C15, 14P99.}}

\author{
Jind{\v r}ich Zapletal\\
University of Florida\\
Academy of Sciences, Czech Republic}

\begin{document}
\maketitle

\begin{abstract}
If $\Gamma$ is a closed Noetherian graph on a $\gs$-compact Polish space with no infinite cliques, it is consistent with the choiceless set theory ZF+DC that $\Gamma$ is countably chromatic and there is no Vitali set.
\end{abstract} 

\section{Introduction}

Chromatic numbers of algebraic and $\gs$-algebraic graphs on Euclidean spaces have been studied extensively in both ZFC and choiceless ZF+DC context \cite{komjath:decomposition, komjath:list, schmerl:avoidable, z:distance}. In this paper, I show it consistent for many such graphs $\Gamma$ that ZF+DC holds, chromatic number of $\Gamma$ is countable, yet there is no Vitali set. The main feature of the graphs exploited here is omission of several simple subgraphs.

\begin{definition}
The \emph{half graph} is the graph on the vertex set $\gw\times 2$, connecting vertices $\langle n, 0\rangle$ and $\langle m, 1\rangle$ if $m<n$ and containing no other edges. A \emph{variation} of the half graph is a graph obtained from the half graph by making vertices $\langle n, 0\rangle$ for $n\in\gw$ either all pairwise connected or all pairwise disconnected, and similarly for vertices $\langle n, 1\rangle$ for $n\in\gw$. The \emph{three-quarter graph} is the graph on the same vertex set, connecting vertices $\langle n, 0\rangle$ and $\langle m, 1\rangle$ if $m\neq n$ and containing no other edges. Variations of the three-quarter graph are defined in the same way.
\end{definition}

\noindent Thus, the half graph and the three-quarter graph have four variations each.

\begin{definition}
A graph $\Gamma$ on a set $X$ is \emph{Noetherian} if it does not contain a variation of the half graph or the three quarter graph as a vertex-induced subgraph.
\end{definition}

\noindent Interesting examples of Noetherian graphs are included in Section~\ref{examplesection}. The precise preservation result obtained in this paper concerns Hamming graphs.

\begin{definition}
\label{hammingdefinition}
The \emph{infinite breadth Hamming graph} $\mathbb{H}_\gw$ is the graph on $\baire$ connecting two points if they differ in exactly one entry. The \emph{diagonal Hamming graph} $\mathbb{H}_{<\gw}$ is the restriction of $\mathbb{H}_\gw$ to the diagonal set $\prod_n(n+1)$.
\end{definition}

\noindent It is not difficult to see that the Hamming graphs are Noetherian.

\begin{theorem}
\label{maintheorem}
Suppose that $\Gamma$ is a closed Noetherian graph on a $\gs$-compact Polish space $X$.

\begin{enumerate}
\item If $\Gamma$ contains no infinite clique, then it is consistent relative to an inaccessible cardinal that ZF+DC holds, the chromatic number of $\Gamma$ is countable while that of $\mathbb{H}_\gw$ is not;
\item if there is a number $n\in\gw$ such that the graph $\Gamma$ contains no clique of cardinality $n$,  then it is consistent relative to an inaccessible cardinal that ZF+DC holds, the chromatic number of $\Gamma$ is countable while that of $\mathbb{H}_{<\gw}$ is not.
\end{enumerate}
\end{theorem}

\noindent The inaccessible cardinal assumption is necessary only to make the proof fit the set-up of geometric set theory \cite{z:geometric} and probably can be dropped. Similarly, the $\gs$-compact assumption, satisfied in the important algebraic examples, is only necessary to evaluate the complexity of the coloring poset in Proposition~\ref{complexityproposition} and probably can be dropped.

In both cases, the conclusion excludes a Vitali set. To see this, let $\{\eps_{n,m}\colon n, m\in\gw\}$ be a collection of pairwise distinct positive rationals with a finite sum, and let $h\colon\baire\to\mathbb{R}$ be the function defined by $h(x)=\gS_n\eps_{n, x(n)}$. The function $h$ is a homomorphism from either of the Hamming graphs to the Vitali equivalence relation, and if $A\subset\mathbb{R}$ were a Vitali set, then the $h$-preimages of $A$ and its rational shifts would show that the chromatic numbers of the Hamming graphs are countable.

Theorem~\ref{maintheorem} is stated in a way which covers many special cases. To state a couple of more specific consequences, for $n\geq 1$ and a set $a$ of positive real numbers let $\Gamma_{na}$ be the graph on $\mathbb{R}^n$ consisting of pairs of points whose Euclidean distance belongs to $a$. 

\begin{corollary}
Let $n\geq 1$ be a number and let $a$ be a countable bounded set of positive reals with $0$ as the only accumulation point. Then it is consistent relative to an inaccessible cardinal that ZF+DC holds, the chromatic number of $\Gamma_{na}$ is countable, yet there is no Vitali set.
\end{corollary}

\noindent This stands in contradistinction with the case in which $a$ is the set of all positive rationals, where the countable chromatic number of $\Gamma_{na}$ yields a Vitali set by the definitions. To prove the corollary, first use Example~\ref{distanceexample} to show that the graph $\Gamma_{na}$ is closed and Noetherian. Theorem~\ref{maintheorem}(1) then proves consistency of ZF+DC plus the chromatic number of $\Gamma_{na}$ is countable while there is no Vitali set.

For the next corollary, let $\langle \eps_n\colon n\in\gw\rangle$ be a sequence of positive real numbers such that $\gS_n (n+1)\eps_n<\infty$. Let $a=\{m\eps_n\colon n\in\gw, m\in n+1\}$.

\begin{corollary}
Let $n\geq 1$ and let $\Gamma$ be an arbitrary algebraic graph on a Euclidean space, without a perfect clique. It is consistent relative to an inaccessible cardinal that ZF+DC holds, the chromatic number of $\Gamma$ is countable, yet the chromatic number of $\Gamma_{1a}$ is uncountable.
\end{corollary}

\noindent This shows that it is possible to color algebraic graphs in general without coloring even quite simple instances of the distance graphs which are $\gs$-algebraic. To prove the corollary, first observe that the Hamming graph $\mathbb{H}_{<\gw}$ can be homomorphically embedded into $\Gamma_{1a}$ by a function $h\colon\prod_n(n+1)\to\mathbb{R}$ defined by $h(x)=\gS_n x(n)\cdot\eps_n$. Now, use Theorem~\ref{boundtheorem} to see that there is a finite bound on the cardinality of $\Gamma$-cliques. Theorem~\ref{maintheorem}(2) then shows the consistency of ZF+DC plus the chromatic number of $\Gamma$ is countable while that of $\mathbb{H}_{<\gw}$ is not--which by the existence of the homomorphism means that the chromatic number of $\Gamma_{1a}$ is uncountable as well.

The techniques of this paper provide much more detailed information about the models obtained than what fits into the statements of the main theorems. However, a number of questions remain open. An affirmative answer to the following question would be a natural strengthening of the main results of this paper.

\begin{question}
Let $\Gamma$ be a closed Noetherian graph on a $\gs$-compact Polish space without an infinite clique. Is it consistent with ZF+DC that $\Gamma$ is countably chromatic, yet there is no linear ordering of the set of all Vitali classes?
\end{question}

To describe the architecture of the paper, in Section~\ref{observationsection} I provide basic insights into closed Noetherian graphs without infinite cliques. In particular, they carry a canonical Noetherian topology, and they are countably chromatic in ZFC. Section~\ref{examplesection} provides examples associated with Euclidean spaces, which are the most interesting from historical point of view. Many quite different examples will doubtless be found in the future. Section~\ref{balancedsection} analyzes a canonical coloring poset used over the Solovay model to add a coloring of closed Noetherian graphs without infinite cliques. Section~\ref{controlsection} discusses the main technical tool to control the generic extension of the Solovay model, namely the finite condition coloring poset known from the work of Todorcevic and others. Finally, Section~\ref{wrapupsection} wraps up the proofs using the technology of \cite[Chapter 11]{z:geometric}.

The notation of the paper follows the set theoretic standard of \cite{jech:newset}. A graph $\Gamma$ on a Polish space $X$ is closed if the relation $\{\langle x, y\rangle\in X^2\colon x\mathrel\Gamma y$ or $x=y\}$ is a closed subset of $X^2$. A topology $\mathcal{T}$ on a set $X$ is Noetherian if there are no infinite sequences of $\mathcal{T}$-closed sets strictly decreasing with respect to inclusion, or equivalently, the intersection of any collection of $\mathcal{T}$-closed sets is equal to the intersection of a finite subcollection. The Vitali equivalence relation on $\mathbb{R}$ connects points $x, y$ if $x-y$ is a rational number; a Vitali set is a subset of $\mathbb{R}$ which intersects each class of the Vitali equivalence relation in exactly one point.

\section{Initial observations}
\label{observationsection}

This section contains basic definitions and facts about Noetherian graphs without infinite cliques. In particular, they carry a canonical Noetherian topology which will be used repeatedly in the paper. I will use the following notation regarding graph neighborhoods throughout.

\begin{definition}
Let $\Gamma$ be a graph on a set $X$.

\begin{enumerate}
\item If $x\in X$ is a vertex, the symbol $\Gamma(x)$ denotes the set $\{y\in X\colon y=x\lor y\mathrel\Gamma x\}$;
\item for a finite set $a\subset X$, $\Gamma(a)$ is the set $\bigcap_{x\in a}\Gamma(x)$;
\item the \emph{$\Gamma$-topology} or the \emph{graph topology} is the smallest topology on $X$ in which all sets $\Gamma(x)$ for $x\in X$ are closed.
\end{enumerate}
\end{definition}

\noindent As is suggested by the terminology, the graph topology of Noetherian graphs is Noetherian. This feature will be used throughout the paper. 

\begin{theorem}
\label{noetheriantheorem}
Let $\Gamma$ be a graph on a set $X$. The following are equivalent:

\begin{enumerate}
\item $\Gamma$ is Noetherian;
\item the $\Gamma$-topology is Noetherian.
\end{enumerate}
\end{theorem}

\begin{proof}
Failure of (1) immediately implies the failure of (2). Let $\pi\colon\gw\times 2\to X$ be an isomorphism of a variation of the half graph to a subgraph of $X$. Let $a_n=\{\pi(m, 0)\colon m\in n\}$, and observe that the sets $\Gamma(a_n)$ all contain all but finitely many points in the set $\{\pi(m, 1)\colon m\in\gw\}$ while their intersection contains none of these points. An identical argument works for an embedding of a variation of the three-quarter graph.

Now, suppose that (1) holds and work to establish (2).  

\begin{claim}
\label{claim1}
There is no sequence $\langle a_n\colon n\in\gw\rangle$ of finite subsets of $X$ such that the sets $\Gamma(a_n)$ strictly decrease with $n$. 
\end{claim}

\begin{proof}
Suppose towards a contradiction that there is such a sequence. Without loss assume that the sets $a_n$ increase with $n\in\gw$, and, erasing needless entries if necessary, assume that there are points $x_n$ such that $a_{n+1}=a_n\cup\{x_n\}$. Note that the points $x_n$ for $n\in\gw$ must be pairwise distinct. For each number $n\in\gw$, let $y_n\in \Gamma(a_n)\setminus\Gamma(x_n)$ be an arbitrary point; note that these points also have to be pairwise distinct and in addition $x_n\neq y_n$. Repeatedly using the Ramsey theorem, find an infinite set $b\subset\gw$ such that 

\begin{itemize}
\item each of the sets $\{x_n\colon n\in b\}$ and $\{y_n\colon n\in b\}$ is either a $\Gamma$-clique or a $\Gamma$-anticlique, and they are disjoint;
\item either for every pair $m<n$ of numbers in $b$, $y_m\mathrel\Gamma x_n$ holds, or for every pair $m<n$ of numbers in $b$, $y_m\mathrel\Gamma x_n$ fails.
\end{itemize}

\noindent Now, let $\pi\colon \gw\to b$ be the increasing enumeration. Define the injection $h\colon\gw\times 2\to X$ by $h(n, 0)=y_{\pi(n)}$ and $h(n, 1)=x_{\pi(n)}$.
If the ``either'' case  in the second item above prevails, then $h$ is an isomorphism of a variation of the three-quarter graph to a subgraph of $\Gamma$. If the ``or'' case prevails, then $h$ is an isomorphism of a variation the half graph to a subgraph of $\Gamma$. Both cases are ruled out by (1). A contradiction.
\end{proof}

\noindent Now, consider the collection $\mathcal{T}$ of all finite unions of sets $\Gamma(a)$ as $a$ ranges over all finite subsets of $X$. 

\begin{claim}
\label{claim2}
There are no infinite strictly descending sequences of sets in $\mathcal{T}$.
\end{claim}

\begin{proof}
This is a standard argument. Towards a contradiction, assume that $\{C_n\colon n\in\gw\}$ is a inclusion-decreasing sequence of sets in $\mathcal{T}$ which does not stabilize. By recursion on $m\in\gw$ build numbers $n_m$ and finite sets $a_m\subset X$ such that

\begin{itemize}
\item $n_0\in n_1\in\dots$;
\item $\Gamma(a_m)\subseteq C_{n_m}$ and the sequence $\{C_n\cap\Gamma(a_m)\colon n\in\gw\}$ does not stabilize;
\item $\Gamma(a_{m+1})$ is a strict subset of $\Gamma(a_m)$.
\end{itemize}

\noindent The base step is subsumed in the recursion step. For the recursion step, suppose that $n_m, a_m$ have been constructed. Since the sequence $\{C_n\cap\Gamma(a_m)\colon n\in\gw\}$ does not stabilize, there must be a number $n_{m+1}>n_m$ such that $C_n\cap\Gamma(a_n)\neq\Gamma(a_n)$. Since the set $C_{n_{m+1}}$ is a finite union of sets of the form $\Gamma(a)$, there must be a finite set $a\subset X$ such that $\Gamma(a)\subseteq C_{n_{m+1}}$ and the sequence $\{C_n\cap\Gamma(a_m)\cap\Gamma(a)\colon n\in\gw\}$ does not stabilize. Set $a_{m+1}=a_n\cup a$ and observe that the recursion step has been successfully performed.

In the end, the sets $\Gamma(a_m)$ contradict the conclusion of Claim~\ref{claim1}. Thus, $\mathcal{T}$ contains no infinite strictly decreasing sequences of sets.
\end{proof}

Now, observe that $\mathcal{T}$ is closed under finite intersections and unions by its definition. It is also closed under arbitrary intersections: the nonexistence of infinite strictly decreasing sequences of sets in $\mathcal{T}$ implies that an intersection of arbitrary collection of sets in $\mathcal{T}$ is equal to an intersection of a finite subcollection. Thus, $\mathcal{T}$ is exactly the collection of closed sets in the $\Gamma$-topology. The Noetherian property of the topology follows immediately from Claim~\ref{claim2}.
\end{proof}

\noindent The next result quantifies the complexity of the graph topology from descriptive point of view. 

\begin{theorem}
\label{definabilitytheorem}
Suppose that $\Gamma$ is a closed Noetherian graph on a $\gs$-compact Polish space $X$. Then the $\Gamma$-topology is analytic.
\end{theorem}

\begin{proof}
This is to say \cite{z:noetherian} that in the usual topology on the space $F(X)$ of all closed subsets of $X$, the collection of all sets closed in the $\Gamma$-topology is analytic. 

Since the space $X$ is $\gs$-compact, the intersection and union functions on $F(X)$ are both Borel, and if $Y$ is a Polish space and $C\subset Y\times X$ is a closed set, the map $y\mapsto C_y$ is a Borel map from $Y$ to $F(X)$ \cite[Section 12.C]{kechris:classical}. It follows that for every $n, k\in\gw$, the map $\pi_{nk}\colon (X^n)^k\to F(X)$ given by $\pi_{nk}(y)=\bigcup_{i\in k}\bigcap_{j\in n}\Gamma(y(i)(j))$ is Borel. Theorem~\ref{noetheriantheorem} shows exactly that the $\Gamma$-topology is the union of the ranges of all functions $\pi_{nk}$, and therefore analytic in $F(X)$.
\end{proof}

\noindent Finally, I show that all closed Noetherian graphs without uncountable cliques are countably chromatic in ZFC.

\begin{theorem}
\label{chromatictheorem}
Let $\Gamma$ be a closed Noetherian graph on a Polish space $X$, without an uncountable clique. The chromatic number of $\Gamma$ is countable.
\end{theorem}

\noindent The proof uses a definition and a proposition which will be of use later.

\begin{definition}
Let $\Gamma$ be a graph on a set $X$.

\begin{enumerate}
\item Let $a\subset X$ be a finite set. Then $\heartsuit(a)$ denotes the set $\{x\in X\colon \forall y\in a\ x=y$ or $x\mathrel\Gamma y$, and $\forall z\in X\ \forall y\in a (z=y\lor z\mathrel\Gamma y)\to (x=z\lor x\mathrel \Gamma z)\}$;
\item a set $A\subset X$ is \emph{$\Gamma$-good} if for every finite subset $a\subset A$, $\heartsuit(a)\subset A$.
\end{enumerate}
\end{definition}

\noindent It is obvious that $\heartsuit(a)$ is a $\Gamma$-clique and that $\Gamma$-goodness is a closure property. In particular, an increasing union of $\Gamma$-good sets is again $\Gamma$-good, and if no uncountable cliques exist in $\Gamma$ then every infinite subset of $X$ can be enclosed in a $\Gamma$-good set of the same cardinality. 

\begin{proposition}
\label{goodproposition}
Let $\Gamma$ be a closed Noetherian graph on a Polish space $X$. Let $A\subset X$ be a $\Gamma$-good set. For every point $x\in X\setminus A$ there is a basic open set $O\subset X$ containing $x$ and containing no elements of $A$ which are $\Gamma$-connected with $x$.
\end{proposition}

\begin{proof}
Suppose that $a\subset A$ is a set of points, all connected to a point $x\in X$ which is an accumulation point of $a$. It will be enough to show that $x\in A$ holds. 

To this end, use Theorem~\ref{noetheriantheorem} to find a finite set $b\subset a$ such that the set $\Gamma(b)$ is as small as possible. Note that $x\in\Gamma(b)$. Moreover, $x$ is $\Gamma$-connected to every point $y\in\Gamma(b)$ distinct from $x$. To see this, use the choice of the set $b$ to observe that the point $y$ is $\Gamma$-connected to every element of $a$. Since the graph $\Gamma$ is closed and $x$ is an accumulation point of $a$, $x\mathrel\Gamma y$ follows.

It follows that $x\in\heartsuit(b)$, therefore $x\in A$ holds by the goodness of the set $A$. The proof is complete.
\end{proof}

\begin{proof}[Proof of Theorem~\ref{chromatictheorem}]
Call a partial $\Gamma$-coloring $c$ \emph{suitable} if its range consists of basic open subsets of $X$ and for each $x\in\dom(c)$, $x\in c(x)$ holds.
By transfinite induction on the cardinality of an infinite $\Gamma$-good set $A\subset X$ prove that if $d$ is a function with domain $A$ assigning to any point $x\in A$ its basic open neighborhood $d(x)$, then there is a suitable coloring $c$ with domain $A$ such that $c(x)\subset d(x)$ holds for every $x\in A$. This will prove the theorem: in the end, one can apply it to $A=X$.

The statement is clear for countable $A$ as the coloring $c$ can in such a case be selected as an injection. Now, suppose that $A$ is a $\Gamma$-good set of uncountable cardinality, such that for all good sets of smaller cardinality the statement is known. Let $d$ be a function with domain $A$ such that for each $x\in A$ the value $d(x)$ is an open neighborhood of $A$. Express $A=\bigcup_{\gb\in\ga}A_\gb$ as an increasing union of $\Gamma$-good sets of smaller cardinality.
For each $\gb\in\ga$ let $d_\gb$ be a function with domain $A_\gb$ such that for every $x\in A_\gb$, $d_\gb(x)$ is an open neighborhood of $x$ which is a subset of $d(x)$ and if $x\notin\bigcup_{\gg\in\gb}A_\gg$, then $d_\gb(x)$ contains no elements of $\bigcup_{\gg\in\gb}A_\gg$ which are $\Gamma$-connected to $x$. This is possible by Proposition~\ref{goodproposition}. By the induction hypothesis, find suitable colorings $c_\gb$ with domain $A_\gb$ such that for every point $x\in A_\gb$, $c_\gb(x)\subset d_\gb(x)$. Now, let $c$ be the function with domain $A$ defined by $c(x)=c_\gb(x)$ where $\gb\in\ga$ is the smallest ordinal such that $x\in A_\gb$. It is not difficult to see that $c$ is a suitable $\Gamma$-coloring of the set $A$ verifying the induction step.
\end{proof}

\section{Initial examples}
\label{examplesection}

The main theorems of the paper need a supply of examples of Noetherian graphs to be meaningful. I concentrate on algebraic graphs and certain special type of $\gs$-algebraic graphs on Euclidean spaces.

\begin{definition}
Let $X$ be a Euclidean space of dimension $n\geq 1$. A graph $\Gamma$ on $X$ is \emph{algebraic} if there is a polynomial $\phi(\bar u, \bar v)$ of $2n$ free variables and real parameters such that for distinct points $x, y\in X$, $x\mathrel\Gamma y$ if and only if $\phi(x, y)=0$.
\end{definition}

\begin{theorem}
\label{boundtheorem}
Let $\Gamma$ be an algebraic graph on a Euclidean space $X$. Then $\Gamma$ is Noetherian, and exactly one of the following occurs:

\begin{enumerate}
\item $\Gamma$ contains a perfect clique;
\item $\Gamma$ is Noetherian, and there is a number $m\in\gw$ such that $\Gamma$ contains no clique of cardinality greater than $m$.
\end{enumerate}
\end{theorem}

\begin{proof}
To prove the Noetherian property of the graph $\Gamma$, work to exclude a variation of the three quarter graph from it (the half graph is treated in the same way). Suppose towards a contradiction that $x_n, y_n\colon n\in\gw$ are vertices in $X$ which induce a copy of a variation of the three quarter graph. Consider the intersection of sets $\Gamma(x_n)$ for $n\in\gw$. This is an intersection of an infinite collection of algebraic sets which contains no points $y_n$ for $n\in\gw$ since $x_n$ is disconnected with $y_n$. However, the intersection of any finite subcollection contains all but finitely many points $x_n$ for $n\in\gw$. This contradicts the Hilbert basis theorem.

(1) clearly implies the failure of (2). To show that the failure of (1) implies a finite bound on the size of $\Gamma$-cliques, I will need a general claim. 

\begin{claim}
There is a number $k$ such that for every finite set $a\subset X$ there is a set $b\subseteq a$ of cardinality at most $k$ such that $\Gamma(b)=\Gamma(a)$.
\end{claim}

\begin{proof}
Write $n$ for the dimension of $X$ and $l$ for the degree of the polynomial defining the graph $\Gamma$. As a quite inefficient estimate, $k=n(2^nl)^n$ will work. To see this, by tree recursion build a tree $T$ and functions $f, g$ on $T$ so that

\begin{itemize}
\item $f(0)=X$ and for every $t\in T$, $f(t)\subseteq X$ is always an irreducible algebraic subset of $X$;
\item for every node $t\in T$, $g(t)$ is some element of $a$ such that $\Gamma(g(t))\cap f(t)\neq f(t)$ if it exists, otherwise $g(t)=!$ and $t$ is a terminal node of $T$;
\item for every node $t\in T$, if $g(t)\in a$ then the set $\{f(s)\colon s$ is an immediate successor of $t$ in $T\}$ lists irreducible components of the algebraic set $\Gamma(g(t))\cap f(t)$ without repetition.
\end{itemize}

\noindent It turns out that the cardinality of the tree $T$ is at most $k$. To see this, work to estimate the depth of the tree and the rate at which it branches. Since the function $f$ maps the tree ordering on $T$ to strict inclusion of irreducible algebraic subsets of $X$, the depth of the tree is at most $n$. By induction on $|t|$, where $t\in T$, argue that the set $g(t)\subseteq X$ is defined by a polynomial of degree at most $2^{|t|}l$. To see that, note that if $p$ is a polynomial defining $g(t)$, then for every immediate successor $s$ of $t$, $g(s)$ is defined by an irreducible factor of $p^2+q^2(f(t))$ where $q(f(t))$ is the polynomial defining $\Gamma(f(t))$, which is of degree at most $l$. Lastly, since the immediate successors of the node $t$ are labeled with irreducible components of $g(t)$ and these are given by irreducible factors of $p^2+q^2(f(t))$, the node $t$ can have at most $2^{|t|+1}l$ many immediate successors. Simple arithmetic then shows that $|T|\leq k$.

In the end, let $b=\rng(f)$ and observe that the set $b$ works: the sets $\Gamma(a)$ and $\Gamma(b)$ are both equal to the union of $f(t)$ as $t$ ranges over all terminal nodes of the tree $T$.
\end{proof}

\noindent Now, suppose that $\Gamma$ has no perfect clique. Let $k\in\gw$ be a number which works as in the claim. Consider the set $B\subset X^k\times X$ defined by $\langle x, y\rangle\in B$ if $y\in\heartsuit(\rng(x))$. The set $B$ is semi-algebraic, and all of its vertical sections are semi-algebraic $\Gamma$-cliques. None of them are uncountable by the assumption on $\Gamma$, and since every semi-algebraic set is either finite or uncountable, all vertical sections of $B$ are finite. Every semi-algebraic set with finite vertical sections enjoys a finite bound on the cardinality of vertical sections \cite[Chapter 3, Lemma 1.7]{vandendries:tame}. Let $m\in\gw$ be such a bound for the cardinality of vertical sections of $B$. I claim there are no $\Gamma$-cliques of cardinality greater than $m$.

To see this, suppose that $a\subset X$ is a $\Gamma$-clique. Let $b\subset a$ be a set of cardinality at most $k$ such that $\Gamma(b)=\Gamma(a)$. Let $x\in X^k$ enumerate, with possible repetitions, all elements of $b$. Clearly, $a\subseteq\heartsuit(b)$ holds, therefore $|a|\leq m$ as desired.
\end{proof}

\begin{example}
Let $C\subset\mathbb{R}^2$ be an irreducible algebraic curve containing the origin, and not equal to a line through the origin. Let $\Gamma_C$ be the graph on $\mathbb{R}^2$ connecting distinct points $x, y$ if $x-y\in C$ or $y-x\in C$. The graph $\Gamma_C$ has no uncountable clique. In fact, there is a number $n\in\gw$ such that $\Gamma_C$ does not contain the bipartite graph $K_{2, n}$ as a subgraph.
\end{example}

\begin{proof}
First argue that $\Gamma_C$ does not contain $K_{2, \gw}$. Suppose towards a contradiction that $x_0, x_1\in \mathbb{R}^2$ are distinct points and $a\subset\mathbb{R}^2$ is an infinite set of points all of which are connected to both $x_0$ and $x_1$. For definiteness, assume that the set $b=\{y\in a\colon y-x_0\in C$ and $y-x_1\in C\}$ is infinite. Then the set $b$ is an infinite subset of $(C+x_0)\cap (C+x_1)$. Since the two algebraic sets in this intersection are irreducible, their intersection is either finite or equal to both. In conclusion, $C+x_0=C+x_1$ holds, in other words $C+(x_0-x_1)=C$. Since $0\in C$, this means that $n(x_0-x_1)\in C$ for any $n\in\gw$, and $C$ has infinite intersection with the line through the origin of direction $x_0-x_1$. An irreducibility argument again shows that $C$ has to be equal to that line, contradicting the initial choice of $C$.

Now, consider the algebraic set $B=\{\langle z, y\rangle\in (\mathbb{R}^2)^2\times\mathbb{R}^2\colon$ the two entries of $z$ are distinct and $y$ is $\Gamma_C$-related to each$\}$. This is a semi-algebraic set with finite vertical sections by the previous paragraph. By \cite[Chapter 3, Lemma 1.7]{vandendries:tame}, there is a number $n\in\gw$ such that all vertical sections of the set $B$ have cardinality at most $n$. This completes the proof.
\end{proof}

\begin{definition}
Let $X$ be a Euclidean space of dimension $n\geq 1$. A graph $\Gamma$ on $X$ is \emph{tight $\gs$-algebraic} if there are algebraic graphs $\Gamma_m$ for $m\in\gw$ on $X$ such that $\Gamma=\bigcup_m\Gamma_m$ and the real numbers $\eps_m=\sup\{d(x, y)\colon x\mathrel\Gamma_m y\}$ tend to zero.
\end{definition}

\begin{theorem}
Let $\Gamma$ be a tight $\gs$-algebraic graph on a Euclidean space $X$. Then $\Gamma$ is Noetherian and exactly one of the following occurs:

\begin{enumerate}
\item $\Gamma$ contains a perfect clique;
\item $\Gamma$ contains no infinite clique.
\end{enumerate}
\end{theorem}

\begin{proof}
Let $\Gamma=\bigcup_m\Gamma_m$ be a witness to the tight $\gs$-algebraicity of $\Gamma$. 
To verify the Noetherian property of $\Gamma$, suppose towards a contradiction that $x_n, y_n\colon n\in\gw$ are vertices in $X$ which induce a copy of a variation of the three quarter graph (the half graph is treated in the same way). Thinning out if necessary I may assume that the set $\{x_n, y_n\colon n\in\gw\}$ is discrete. It follows that for each $n\in\gw$ there is a number $m_n\in\gw$ such that the points $x_n$ are $\bigcup_{m\in m_n}\Gamma_m$-connected to points $y_k$ for all $k\neq n$. Consider the intersection of all sets $\bigcup_{m\in m_n}\Gamma_m(x_n)$ for $n\in\gw$. This is an intersection of an infinite collection of algebraic sets which contains no points $y_n$ for $n\in\gw$. The intersection of any finite subcolection contains all but finitely many points $y_n$ for $n\in\gw$. This contradicts the Hilbert basis theorem.

Clearly (1) implies the failure of (2). Now, suppose that (2) fails and work to confirm (1).
Let $\{x_n\colon n\in\gw\}$ is an infinite $\Gamma$-clique. I will produce a number $m\in\gw$ such that an infinite subset of this clique forms a $\Gamma_m$-clique. Then, an application of Theorem~\ref{boundtheorem} provides a perfect $\Gamma_m$-clique, therefore a perfect $\Gamma$-clique.

Thinning out the clique if necessary, assume that it is discrete. By recursion on $k\in\gw$ build numbers $n_k, m_k$, and infinite sets $b_k\subset\gw$ such that

\begin{itemize}
\item $n_{k+1}>n_k$, $b_{k+1}\subset b_k$;
\item $n_{k}\in b_k$;
\item $x_{n_k}\mathrel \Gamma_{m_k}x_n$ for all $n\in b_{k+1}$.
\end{itemize}

\noindent To start, let $n_0=0$ and $b_0=\gw$. Since $x_0$ is an isolated point of the clique, the tightness condition on the graph $\Gamma$ implies that there are only finitely many numbers $m$ such that $x_0\Gamma_m x$ holds for some $x\neq x_0$ in the clique. It follows that there must be a number $m_0$ such that the set $b_1=\{n\in\gw\colon x_0\mathrel\Gamma_{m_0}x_n\}$ is infinite. The recursion step is performed in a similar way.

Now, observe that the set $\{m_k\colon k\in\gw\}$ must be finite. Otherwise, consider the intersection $\bigcap_k \Gamma_{m_k}(x_{n_k})$. This is an intersection of of an infinite collection of algebraic sets. It contains no elements of the set $\{x_{n_k}\colon k\in\gw\}$ since all elements of this set are isolated in it and the numbers $m_k\in\gw$ grow arbitrarily large. Intersection of any finite subcollection always contains all but finitely many of the set $\{x_{n_k}\colon k\in\gw\}$. This contradicts the Hilbert basis theorem.

It is then possible to find a number $m\in\gw$ such that the set $c=\{k\in\gw\colon m_k=m\}$ is infinite. The set $\{x_{n_k}\colon k\in c\}$ is an infinite $\Gamma_m$-clique as desired.
\end{proof}

\begin{example}
\label{distanceexample}
If $a\subset\mathbb{R}$ is a bounded countable set of positive reals with $0$ as the only accumulation point, the graph $\Gamma$ on a Euclidean space $X$ connecting points whose distance belongs to the set $a$ is tight $\gs$-algebraic. It contains no infinite clique.
\end{example}

\begin{proof}
Suppose towards a contradiction that $\{x_n\colon n\in\gw\}$ is an infinite $\Gamma$-clique. Thinning out if necessary, assume that the clique is discrete. For every number $n\in\gw$, there must be a finite set $b_n\subset a$ such that the distance of the point $x_n$ from any point $x_m$ for $m\neq n$ belongs to the set $b_n$. Let $A_n\subset X$ be the algebraic set of all points in $X$ whose distance from $x_n$ belongs to the finite set $b_n$. The intersection of all sets $A_n$ for $n\in\gw$ contains no elements of the clique since $x_n\notin A_n$. On the other hand, intersection of any finite subcollection contains all but finitely many elements of the clique. This contradicts the Hilbert basis theorem.
\end{proof}

\section{A balanced coloring poset}
\label{balancedsection}

This section contains a description of a canonical poset adding a coloring to a closed Noetherian graph without an uncountable clique. The poset is balanced in the sense of \cite[Chapter 5]{z:geometric}. Its further preservation properties will be proved in Section~\ref{wrapupsection}.

\begin{definition}
\label{coloringposetdefinition}
Let $\Gamma$ be a closed Noetherian graph on a $\gs$-compact Polish space $X$ without an uncountable clique. The \emph{coloring poset} $P_\Gamma$ consists of countable partial $\Gamma$-colorings $p$ such that 

\begin{enumerate}
\item $\dom(p)$ is a countable $\Gamma$-good subset of $X$;
\item $\rng(p)$ consists of basic open subsets of $X$ and for each $x\in\dom(p)$, $x\in p(x)$ holds.
\end{enumerate}

\noindent The ordering is defined by $q\leq p$ if $p\subseteq q$ and for every point $x\in\dom(q\setminus p)$, the set $q(x)$ contains no elements of $\dom(p)$ which are $\Gamma$-connected to $x$.
\end{definition}

\noindent Verification of the key properties of the poset $P_\Gamma$ proceeds by a series of propositions.

\begin{proposition}
$P_\Gamma$ is a $\gs$-closed transitive relation.
\end{proposition}

\begin{proof}
The transitivity is clear. If $\langle p_i\colon i\in\gw\rangle$ is a descending sequence of conditions in $P_\Gamma$, then $\bigcup_ip_i$ is their common lower bound.
\end{proof}

\begin{proposition}
\label{thirdp}
Suppose that $a\subset P_\Gamma$ is a finite set. The following are equivalent:

\begin{enumerate}
\item $a$ has a common lower bound in $P$;
\item for every point $x\in X$, $a$ has a common lower bound in $P$ which contains $x$ in its domain;
\item $\bigcup a$ is a function and for any two distinct conditions $p_0, p_1\in a$ and any two $\Gamma$-connected points $x_0\in\dom(p_0\setminus p_1)$ and $x_1\in\dom(p_1\setminus p_0)$, the sets $p_0(x_0)$ and $p_1(x_1)$ do not contain $x_1$ and $x_0$ respectively.
\end{enumerate}
\end{proposition}

\noindent In particular, a finite subset of $P_\Gamma$ has a common lower bound if and only if it consists of pairwise compatible conditions.

\begin{proof}
(2) implies (1) which implies (3) by the definition of the ordering. To show that (3) implies (1), fix the finite set $a\subset P_\Gamma$ and a point $x\in X$, assume that (3) holds, and work to find a common lower bound of $a$ which contains $x$ in its domain.

Let $b\subset X$ be a $\Gamma$-good countable set containing $\dom(p)$ for every $p\in a$ and the point $x$ as well. I will produce a lower bound $q$ of $a$ such that $b=\supp(q)$. To this end, let $d=b\setminus\bigcup_{p\in a}\dom(p)$. For every point $y\in d$, use Proposition~\ref{goodproposition} to find an open neighborhood $O_y\subset X$ of $y$ such that $O_y$ contains no point in $\bigcup_{p\in a}\dom(p)$ which is $\Gamma$-connected to $y$. Then, find an injection $r$ from $d$ to basic open subsets of $X$ such that for every point $y\in d$, $r(y)\subset O_y$ and $y\in r(y)$ holds. It will be enough to show that $q=r\cup\bigcup_{p\in a}p$ is a common lower bound of the set $a$.

First of all, argue that $q$ is a $\Gamma$-coloring. To see this, suppose that $x_0, x_1\in\dom(q)$ are distinct $\Gamma$-connected points. There are several cases.

\noindent\textbf{Case 1.} $x_0, x_1\in d$. In this case, $q(x_0)\neq q(x_1)$ since $q\restriction d$ is an injection.

\noindent\textbf{Case 2.} Exactly one point among $x_0, x_1$, say $x_0$ belongs to $d$. By the choice of the set $O_{x_0}$, $x_1\notin q(x_0)$ holds. At the same time, $x_1\in q(x_1)$ holds and therefore $q(x_0)\neq q(x_1)$ as desired.

\noindent\textbf{Case 3.} Neither $x_0$ nor $x_1$ belongs to $d$. Pick conditions $p_0, p_1\in a$ such that $x_0\in\dom(p_0)$ and $x_1\in\dom(p_1)$ holds. If either $x_0\in\dom(p_1)$ or $x_1\in\dom(p_0)$ holds, then $x_0, x_1$ receive distinct colors since each $p_0, p_1$ is separately a $\Gamma$ coloring. Otherwise, it must be the case that $x_0\in\dom(p_0\setminus p_1)$ and $x_1\in\dom(p_1\setminus p_0)$ holds, and then $q(x_0)\neq q(x_1)$ holds by (3) and the definition of the ordering $P_\Gamma$.

Second, show that for each $p\in a$, $q\leq p$ holds. To this end, let $y\in \dom(q\setminus p)$ be an arbitrary point; the value $q(y)$ must not contain any point $z\in\dom(p)$ which is $\Gamma$-connected to $y$. This is clear if $y$ belongs to the domain of some other condition in $a$ by (3). Otherwise, $y\in d$ holds and then $z\notin O_y$ and $z\notin q(y)$ holds by the choice of the set $O_y$. The proof is complete.
\end{proof}

\begin{proposition}
\label{complexityproposition}
$P_\Gamma$ is a Suslin forcing.
\end{proposition}

\begin{proof}
The complexity calculation starts with an easy initial observation.

\begin{claim}
For every $n\in\gw$ the set $B_n\subset X^n\times X$ of all pairs $\langle y, x\rangle$ such that $x\in\heartsuit(\rng(y))$ is Borel.
\end{claim}

\begin{proof}
To show that the set $B_n$ is Borel, for every relatively open set $O\subset X$ define $C_O=\{y\in X^n\colon\exists x\in O\ \forall z\in\rng(y)\ x\mathrel \Gamma z\lor x=z\}$. Since the set $O\subset X$ is $K_\gs$, a compactness argument shows that $C_O\subset X^n$ is $K_\gs$ as well. Then $\langle y, x\rangle\in B_n$ if $\forall z\in\rng(y)\ x\mathrel\Gamma z\lor x=z$ and for every pair $O_0, O_1$ of disjoint basic open subsets of $X$ such that $O_0\times O_1\cap\Gamma=0$, either $x\notin O_0$ or $y\notin C_{O_1}$. This presents $B_n$ as a Borel set.
\end{proof}

\noindent Now, since vertical sections of the set $B_n$ are $\Gamma$-cliques, they are countable. By the Lusin--Novikov theorem, each set $B_n$ is a union of graphs of countably many Borel functions. It is clear then that conditions are exactly those $\Gamma$-colorings $p$ with countable domain such that the domain is closed under all said Borel functions, and for each $x\in\dom(p)$, $p(x)\subset X$ is a basic open neighborhood of $x$. This shows that the set of conditions in $P_\Gamma$ is Borel.

The ordering on $P_\Gamma$ is obviously a Borel set. Proposition~\ref{thirdp} provides a Borel characterization of compatibility of conditions in $P_\Gamma$, completing the proof.
\end{proof}

\begin{proposition}
$P_\Gamma$ forces the union of the generic filter to be a total $\Gamma$-coloring.
\end{proposition}

\begin{proof}
It is only necessary to show that for every condition $p\in P_\Gamma$ and every $x\in X$ there is a condition $q\leq p$ such that $x\in\dom(q)$. This follows from Proposition~\ref{thirdp} applied to the set $a=\{p\}$.
\end{proof}

\noindent It is now time to prove the instrumental amalgamation property of the coloring poset $P_\Gamma$. Recall \cite{z:noetherian} the following notions.

\begin{definition}
\begin{enumerate}
\item Let $\mathcal{T}$ be an analytic Noetherian topology on a $K_\gs$ Polish space $X$. If $M$ is a transitive model of ZFC containing the code for $\mathcal{T}$ and $A\subset X$ is a set, the symbol $C(M, A)$ denotes the smallest $\mathcal{T}$-closed set coded in $M$ which contains $A$ as a subset.
\item Generic extensions $V[H_0], V[H_1]$ are \emph{mutually Noetherian} if for every analytic Noetherian topology $\mathcal{T}$ on a $K_\gs$ Polish space $X$ coded in the ground model $V$ and for every set $A_1\subset X$ in $V[H_1]$, $C(V, A_1)=C(V[H_0], A_1)$ holds, and vice versa: for every set $A_0\subset X$ in $V[H_0]$, $C(V, A_0)=C(V[H_1], A_0)$ holds.
\end{enumerate}
\end{definition}

\noindent For example, mutually generic extensions are mutually Noetherian. 

\begin{definition}
Let $P$ be a Suslin poset.

\begin{enumerate}
\item A pair $\langle Q, \tau\rangle$ is \emph{Noetherian balanced} if $Q\Vdash\tau\in P$ and for every mutually Noetherian pair $V[H_0]$ $V[H_1]$ of generic extensions, all filters $G_0\subset Q$ in $V[H_0]$ and $G_1\subset Q$ in $V[H_1]$ generic over $V$, and conditions $p_0\leq \tau/G_0$ in $V[H_0]$ and $p_1\leq\tau/G_1$ in $V[H_1]$, the conditions $p_0, p_1\in P$ have a common lower bound. 
\item The poset $P$ is Noetherian balanced if for every condition $p\in P$ there is a Noetherian balanced pair $\langle Q, \tau\rangle$ such that $Q\Vdash\tau\leq \check p$.
\end{enumerate}
\end{definition}

\noindent Thus, Noetherian balance is a strengthening of the usual balance of Suslin posets \cite[Chapter 5]{z:geometric}. It is an amalgamation tool which for example rules out nonprincipal ultrafilters on $\gw$ in $P$-extensions of the choiceless Solovay model \cite{z:noetherian}. It is irrelevant for the main theorems of this paper, but it will be used in future work.

\begin{proposition}
\label{balanceproposition}
The poset $P_\Gamma$ is Noetherian balanced.
\end{proposition}

\begin{proof}
Let $p\in P_\Gamma$ be a condition. I will show that there is a total $\Gamma$-coloring $c$ on $X$ such that $p\subset c$ and for every $x\in\dom(c\setminus p)$, the value $c(x)$ is a basic open subset of $X$ which contains no elements of $\dom(p)$ $\Gamma$-connected to $x$. Then I will show that the pair $\langle \coll(\gw, \mathbb{R}), \check c\rangle$ is a Noetherian balanced pair. This will prove the proposition.

To find $c$, first for every point $x\in X\setminus\dom(p)$ find an open neighborhood $d(x)\subset X$ containing $x$ and no points of $\dom(p)$ which are $\Gamma$-connected to $x$. This is possible by Proposition~\ref{goodproposition}. By the proof of Theorem~\ref{chromatictheorem}, there is a total $\Gamma$-coloring $e$ which to each point $x\in X$ assigns a basic open subset $e(x)\subset d(x)$ containing $x$. Then define $c$ by $c(x)=p(x)$ if $x\in\dom(p)$ and $c(x)=d(x)$ if $x\notin\dom(p)$; this coloring $c$ is as required.

It is clear that $\coll(\gw, \mathbb{R})\Vdash \check c\leq\check p$. Suppose now that $V[H_0], V[H_1]$ are mutually Noetherian extensions of $V$, each containing respective condition $p_0\leq c$ and $p_1\leq c$; I must show that $p_0, p_1$ are compatible. This means that item (3) of Proposition~\ref{thirdp} must be verified for $a=\{p_0, p_1\}$. Suppose that $x_0\in\dom(p_0\setminus c)$ and $x_1\in\dom(p_1\setminus c)$ are $\Gamma$-connected points, and towards a contradiction assume that (e.g.) $x_1\in p_0(x_0)$. The set $A_1=\{y\in X\colon y=x_0\lor y\mathrel\Gamma x_0\}$ is closed in the $\Gamma$-topology, coded in $V[H_0]$, and contains $x_1$. By the Noetherian assumption, the smallest closed in the $\Gamma$-topology set $B_1$ coded in $V$ containing $x_1$ is a subset of $A_1$. By Mostowski absoluteness between $V$ and $V[H_1]$, $B_1$ contains a ground model point $y$ in the basic open set $p(x_0)$. However, this contradicts the assumption that $p_0\leq c$ holds.
\end{proof}

\begin{corollary}
Let $\Gamma$ be a closed Noetherian graph on a $\gs$-compact Polish space without an uncountable clique. In the $P_\Gamma$-extension of the Solovay model,

\begin{enumerate}
\item there are no discontinuous homomorphisms between Polish groups;
\item there is no nonprincipal ultrafilter on $\gw$;
\item the Lebesgue null ideal is closed under well-ordered unions.
\end{enumerate}
\end{corollary}

\begin{proof}
Item (1) follows from the balance and the $3,2$-centeredness of the poset $P_\Gamma$ as proved in \cite[Theorem 13.2.1]{z:geometric}. (2) follows from (1), since a nonprincipal ultrafilter $U$ on $\gw$ yields a discontinuous homomorphism from the Cantor group $\cantor$ to $2$ assigning to any point $x\in\cantor$ its prevailing value. However, (2) also follows from the Noetherian balance of the poset $P_\Gamma$ by \cite{z:noetherian}. (3) follows from the Noetherian balance of $P_\Gamma$ again \cite{z:noetherian}.
\end{proof}

\noindent Properties of the $P_\Gamma$ extension of the Solovay model pertaining to countable Borel equivalence relations must be checked by the methods of the following sections.

\section{The control poset}
\label{controlsection}

This section discussed the main technical tool used to control the generic extension of the Solovay model by the balanced coloring poset $P_\Gamma$ obtained in Section~\ref{balancedsection}.

\begin{definition}
\label{controldefinition}
Let $\Gamma$ be a closed graph on a Polish space $X$. The \emph{$\Gamma$-control poset} $Q_\Gamma$ is the poset of all finite partial $\Gamma$-colorings $q\colon X\to \gw$ ordered by reverse inclusion.
\end{definition}

\noindent It is a well-known result of Todorcevic \cite[Proposition 1]{todorcevic:examples} that the control poset is c.c.c.\ if and only if the graph $\Gamma$ does not contain a perfect clique. I need to precisely quantify the nature of c.c.c.\ of the control posets in case that the graph $\Gamma$ is closed, Noetherian, and contains no infinite cliques. The following common parlance will be useful throughout.

\begin{definition}
Let $\Gamma$ be a closed graph on a Polish space $X$.
\begin{enumerate}
\item A \emph{location} is a pair $\langle a, f\rangle$ where $a$ is a finite collection of pairwise disjoint basic open subsets of $X$ and $f\colon a\to\gw$ is a function such that for any two distinct open sets $O_0, O_1\in a$, either $(O_0\times O_1)\cap\Gamma=0$ or  $f(O_0)\neq f(O_1)$ holds;
\item a function $q$ such that $\dom(q)$ is a selector in $a$ and for every $x\in\dom(q)$ $q(x)=f(O)$ for the unique point $O\in a$ containing $x$ is \emph{at location} $\langle a, f\rangle$;
\item if a location $\langle a, f\rangle$ is given, for any condition $q$ at this location and any $O\in a$ the symbol $q(O)$ denotes the unique point in the domain of $q$ which belongs to $O$.
\end{enumerate}
\end{definition}

\noindent  Note that every function as in (2) is a condition in the control poset $Q_\Gamma$. The first theorem evaluates the descriptive complexity of the control poset. 

\begin{definition}
A c.c.c.\ poset $Q$ is \emph{very Suslin} if 

\begin{enumerate}
\item $Q$ is Suslin. That is, in an ambient Polish space the set $Q$ is analytic, and the relations of ordering and incompatibility are analytic as well;
\item the set $\{d\in Q^\gw\colon\rng(d)\subset Q$ is a predense set$\}$ is analytic.
\end{enumerate}
\end{definition}

\noindent Very Suslin c.c.c.\ forcings do not add dominating reals \cite{Sh:711}, and their complexity is preserved under finite support iterations of countable length.

\begin{theorem}
\label{verysuslintheorem}
If $\Gamma$ is a closed Noetherian graph on a Polish $\gs$-compact space $X$ without an infinite clique, then the poset $Q_\Gamma$ is very Suslin.
\end{theorem}

\begin{proof}
Suslinness of $Q_\Gamma$ follows immediately from the definitions. To start, use Kuratowski--Ryll-Nardzewski theorem \cite[Theorem 12.13]{kechris:classical} to find Borel functions $f_n\colon F(X)\to X$ for $n\in\gw$ such that for every nonempty closed set $C\subset X$, the set $\{f_n(C)\colon n\in\gw\}$ is a dense subset of $C$. Now, let $d\in\mathbb{Q}_\Gamma^\gw$ be a sequence of conditions. Write $b=\bigcup\{\dom(q)\colon q\in\rng(d)\}$ and $c=b\cup\bigcup\{f_n(\Gamma(a))\colon n\in\gw, a\in [b]^{<\aleph_0}\}$. I claim that $\rng(d)\subset Q_\Gamma$ is predense if and only if every condition whose domain is a subset of $c$ is compatible with some element of $\rng(d)$.

First, note that this will show that the set $\{d\in Q^\gw\colon\rng(d)\subset Q$ is a predense set$\}$ is analytic. The universal quantifier on the right side of the equivalence is restricted to a countable subset of $Q_\Gamma$ which is obtained in a Borel way from $d$.

Now, the left-to-right implication is immediate. The right-to-left implication is proved in the contrapositive. Suppose that  
the set $\rng(d)$ is not predense and work to find a condition which is incompatible to all elements of $\rng(d)$, whose domain is a subset of $c$. Let $q\in Q_\Gamma$ be any condition incompatible with all elements of $\rng(d)$.  For each point $x\in\dom(q)\setminus b$, let $b'\subset b$ be a finite set of points $\Gamma$-connected to $x$ and such that the set $C_x=\Gamma(b')$ is as small as possible. Such a set is possible to find by the Noetherian assumption and Theorem~\ref{noetheriantheorem}. 

Now, let $\langle a, f\rangle$ be any location of the condition $q$. Define a condition $r$ at the same location using the following description.

\begin{itemize}
\item if $O\in a$ is such that $q(O)\in b$ then let $r(O)=q(O)$;
\item if $O\in a$ is such that $x=q(O)\notin b$ and there is a point $y\in C_x\cap O\cap b$ which is $\Gamma$-disconnected from $x$, then let $r(O)=y$;
\item finally, suppose that $O\in a$ is such that $x=q(O)\notin b$ and all points of $C_x\cap O\cap b$ are $\Gamma$-connected to $x$. In such a case, $C_x\cap O\cap b$ is a $\Gamma$-clique by the minimal choice of $C_x$. By the initial assumptions on the graph $\Gamma$, this clique is finite. Note that $x\in C_x$ holds and $x$ does not belong to the finite set $C_x\cap O\cap b$, simply because $x\notin b$ holds. Since the set $c\cap C_x$ is dense in $C_x$, there is a point $y\in c\cap C_x\cap O$ which does not contain any elements of $C_x\cap O\cap b$. Let $r(O)=y$.
\end{itemize}

\noindent To complete the proof, argue that $r$ is incompatible with all conditions in $\rng(d)$. To do that, suppose that $s\in\rng(d)$ is a condition. Then $s, q$ are incompatible, and there must be points $x\in\dom(q)$ and $z\in\dom(s)$ such that either $x=z$ and $q(x)\neq s(z)$, or $x\mathrel\Gamma z$ and $q(x)=s(z)$. The discussion now splits into cases.

If $x\in b$ then $s, r$ are incompatible by virtue of the same points $x, z$ since $q\restriction b\subseteq r\restriction b$ by the first item above. If $x\notin b$, then let $O\in a$ be the unique open set which contains it, and let $y=r(O)$. No matter whether the second or third item above occurred, the minimality of $C_x$, the fact that $y\in C_x$, and the fact that $z\in b$ guarantee that $z=y$ or $z\mathrel\Gamma y$ holds. It will be enough to rule out the possibility that $z=y$ holds, since then $s, r$ are incompatible by virtue of the points $y, z$. Now, if the second item above occurred for $O$, $z=y$ is impossible since $y\mathrel\Gamma x$ fails in that case while $z\mathrel\Gamma x$ holds. If the third item above occurred for $O$, then $y\notin b$ and $y$ must again be distinct from $z$ since $z\in b$ holds. The proof is complete.
\end{proof}

\noindent Perhaps more importantly, the Noetherian property of the graph $\Gamma$ has implications for the centeredness properties of the control poset $Q_\Gamma$. Fine distinctions are important here, and we restate the definitions of \cite[Chapter 11]{z:geometric}:

\begin{definition}
Let $Q$ be a poset.

\begin{enumerate}
\item A set $A\subset Q$ is \emph{liminf-centered} if for every infinite set $a\subset A$ there is a condition $q\in Q$ which forces the generic filter to contain infinitely many elements of $a$. 
\item If $Q$ is a Suslin poset, it is \emph{Suslin-$\gs$-liminf-centered} if there are analytic liminf-centered sets $A_n\subset Q$ for $n\in\gw$ such that $Q=\bigcup_nA_n$.
\end{enumerate}
\end{definition}

\noindent Note that a liminf centered set cannot contain an infinite antichain, so $\gs$-liminf centeredness implies c.c.c.

\begin{theorem}
\label{liminfcenteredtheorem}
Suppose that $\Gamma$ is a closed Noetherian graph on a Polish space $X$, without an infinite clique. Then the control poset $Q_\Gamma$ is Suslin-$\gs$-liminf centered.
\end{theorem}

\begin{proof}
Let $\langle a, f\rangle$ be a location. It is enough to show that the set $A$ of all conditions at that location is liminf-centered.
To this end, the following abstract claim will be useful.

\begin{claim}
Let $b\subset X$ be an infinite set. There is an infinite set $c\subset b$ such that each element of $X$ is either $\Gamma$-connected with only finitely many elements of $c$ or with all elements of $c$.
\end{claim}

\begin{proof}
First use the Ramsey theorem to shrink the set $b$ if necessary to a $\Gamma$-anticlique. Use Theorem~\ref{noetheriantheorem} to find an inclusion-minimal set $C\subset X$ in the $\Gamma$-graph topology such that $c=C\cap b$ is infinite. I claim that the set $c$ works.

Indeed, suppose that $x\in X$ is a point. If $x\in c$, then $x$ is $\Gamma$-disconnected from all other elements of $c$ since $c$ is a $\Gamma$-anticlique. If $x\notin c$ and $x$ is $\Gamma$-connected to infinitely many elements of $c$, then $C=C\cap\Gamma(x)$ by the minimal choice of $C$ and therefore $x$ is $\Gamma$-connected to all elements of $c$. The claim follows.
\end{proof}

\noindent Now, let $A=\{q_n\colon n\in\gw\}$ be an infinite collection of conditions at location $\langle a, f\rangle$. Shrinking the collection if necessary, it is possible to find a partition $a=a_0\cup a_1$ such that for every $O\in a_0$ the points $q_n(O)$ for $n\in\gw$ are all the same, while for every set $O\in a_1$ the points $q_n(O)$ for $n\in\gw$ are pairwise distinct. Using the claim repeatedly, it is possible to further shrink the set $A$ so that for every $O\in a_1$, every point in $X$ is $\Gamma$-connected to only finitely many points among $\{q_n(O)\colon n\in\gw\}$ or to all of them. It will be enough to show that $q_0$ forces that the generic filter contains infinitely many conditions in the set $A$.

To see this, suppose that $r\leq q_0$ is a condition; it will be enough to show that $r$ is compatible with all but finitely many conditions in $A$. Indeed, suppose that $n\in\gw$ is large enough so that for every point $x\in\dom(r)$ and every $O\in a_1$, $q_n(O)\neq x$ and if $x$ is $\Gamma$-connected to only finitely many elements of $q_n(O)$, then it is not $\Gamma$-connected to $q_n(O)$. It will be enough to show that $r$ is compatible with $q_n$. For this, note that $q_n\cup r$ is a function by the choice of the number $n$. To show that $q_n\cup r$ is a $\Gamma$-coloring, suppose that $x\in\dom(r)$ and $O\in a_1$ is a set and write $y=q_n(O)$. I must show that $x\mathrel\Gamma y$ implies $r(x)\neq q_n(y)$. To see this, note that $x\mathrel\Gamma y$ implies that $x$ is $\Gamma$-connected with all points in the set $\{q_i(O)\colon i\in\gw\}$ by the choice of the number $n$. In particular, $x$ is $\Gamma$-connected to $z=q_0(O)$, and since $r\leq q_0$ holds, $r(x)\neq q_0(z)$. Now, the conditions $q_0$ and $q_n$ are at the same location and therefore $q_0(z)=q_n(y)$. It follows that the values $r(x)$ and $q_n(y)$ are distinct as required.
\end{proof}

\begin{definition}
\label{ramseydefinition}
Let $Q$ be a poset.

\begin{enumerate}
\item A set $A\subset Q$ is \emph{Ramsey centered} if for every $m\in\gw$ there is $n\in\gw$ such that every collection of $n$ many elements of $A$ contains a subcollection of cardinality $m$ which has a common lower bound.
\item If $Q$ is a Suslin poset, it is \emph{Suslin-$\gs$-Ramsey-centered} if there are analytic Ramsey-centered sets $A_n\subset Q$ for $n\in\gw$ such that $Q=\bigcup_nA_n$.
\end{enumerate}
\end{definition}

\noindent Note that a Ramsey-centered set cannot contain an infinite antichain, so $\gs$-Ramsey centeredness implies c.c.c.\ again.

\begin{theorem}
\label{ramseycenteredtheorem}
Suppose that $m\in\gw$ is a number and $\Gamma$ is a closed graph on a Polish space $X$, containing no cliques of cardinality $m$. Then the control poset $Q_\Gamma$ is Suslin-$\gs$-Ramsey centered.
\end{theorem}

\noindent Note that this theorem does not refer to any Noetherian assumption; indeed, its proof is much easier than that of Theorem~\ref{liminfcenteredtheorem}.

\begin{proof}
Let $\langle a, f\rangle$ be a location. It is enough to show that the set $A$ of all conditions at that location is Ramsey-centered. To this end, let $m\in\gw$ be a number; increasing $m$ if necessary assume that $\Gamma$ contains no cliques of cardinality $m$. Writing $s$ for the cardinality of the set $A$, any natural number $k$ such that $k\to (m)^2_{s+1}$ witnesses Ramsey centeredness of the set $A$ for the number $m$. To see this, let $q_i\colon i\in k\rangle$ be a collection of conditions in $A$. Define a map $c$ with domain $[k]^2$ by setting $c(i, j)=O$ if the conditions $q_i, q_j$ are incompatible and this fact is witnessed by the open set $O\in a$ in the sense that the unique elements $x_i, x_j$ in $\dom(q_i)\cap O$ and $\dom(q_j)\cap O$ are $\Gamma$-connected. Define $c(i, j)=!$ if the conditions $q_i, q_j$ are compatible. 

By the choice of the number $k$ there is a set $b\subset k$ of cardinality $m$ homogeneous for the partition $c$. The homogeneous color cannot be a set $O\in a$, since then the unique points in $\dom(q_i)\cap O$ for $i\in b$ would form a $\Gamma$-clique of cardinality $m$, contradicting the initial choice of $m$. Thus, the homogeneous color is $!$ and the set $\{q_i\colon i\in b\}$ consists of pairwise compatible conditions. Then $\bigcup_{i\in b}q_i$ is a common lower bound of this set, confirming Ramsey-centeredness of the set $A$.
\end{proof}

\section{Wrapping up}
\label{wrapupsection}

The theorems in the introduction are now obtained via the technologies of \cite[Chapter 11]{z:geometric}.

\begin{definition}
\label{liminfcontroldefinition}
Let $P$ be a Suslin poset. Say that $P$ has \emph{liminf control} if provably in ZFC, there is a definable map which assigns to each condition $p\in P$ objects $R, \pi, Q, \tau$ so that

\begin{enumerate}
\item $R$ is a forcing and $Q, \tau$ are $R$-names;
\item $\pi\colon R\to\mathrm{Ord}$ is a function such that preimages of singletons are liminf-centered;
\item $R\Vdash \langle Q, \tau\rangle$ is a balanced pair in $P$ and $Q\Vdash\tau\leq \check p$.
\end{enumerate}
\end{definition}

\noindent In other words, there may not be a definable balanced pair for $P$, but there is a definable and suitably centered way of forcing such a pair. The following is a complementary concept from descriptive set theory.

\begin{definition}
\textnormal{\cite[Definition 11.1.5]{z:geometric}}
An analytic graph on a Polish space $X$ has Borel $\gs$-finite clique number if there are Borel sets $B_n\subset X$ for $n\in\gw$ such that $\bigcup_nB_n=X$ and no set $B_n$ contains an infinite clique.
\end{definition}

\noindent A good example of a Borel graph with uncountable Borel $\gs$-finite clique number is the Hamming graph $\mathbb{H}_\gw$ of infinite breadth from Definition~\ref{hammingdefinition}. The following is proved in \cite[Theorem 11.7.5]{z:geometric}.

\begin{fact}
\label{liminffact}
Suppose that $P$ is balanced and has liminf control. Suppose that $\Delta$ is an analytic graph with uncountable Borel $\gs$-finite clique number. Let $\kappa$ be an inaccessible cardinal and $W$ be the derived choiceless Solovay model. In the $P$-extension of $W$, $\Delta$ has uncountable chromatic number.
\end{fact}

\begin{proposition}
\label{liminfcontrolprop}
Let $\Gamma$ be a closed Noetherian graph on a $\gs$-compact Polish space $X$ without infinite cliques. Then the coloring poset $P_\Gamma$ has liminf control.
\end{proposition}

\begin{proof}
Let $p\in P_\Gamma$ be a condition. Let $R$ be the finite support iteration of the control poset $Q_\Gamma$ of Definition~\ref{controldefinition} of length $\gw_1$. Since the control poset is Suslin-$\gs$-liminf-centered by Theorem~\ref{liminfcenteredtheorem}, the iteration $R$ is definable and definably liminf-centered as in Definition~\ref{liminfcontroldefinition} as proved in \cite[Proposition 11.7.3]{z:geometric} I will produce a definition of an $R$-name for a total coloring of $X$ stronger than $p$. This will prove the proposition, since by the (proof of) Proposition~\ref{balanceproposition} such a coloring defines a balanced pair.

First, note that the iteration $R$ produces a transfinite sequence $\langle M_\ga\colon\ga\in\gw_1\rangle$ of forcing extensions, where we put $M_0$ to be the ground model. By a c.c.c.\ argument, the space $X$ in the $R$-extension is a subset of the union $\bigcup_\ga M_\ga$. The iteration $R$  also produces a transfinite sequence $\langle c_\ga\colon \ga\in\gw_1\rangle$ of partial $\Gamma$-colorings. Namely, each coloring $c_\ga$ is the union of the generic filter on the $\ga$-th iterand of $R$, and as such its domain is exactly $X\cap M_\ga$. The colorings do not cohere in any way. It is necessary to stitch them together in a definable way to produce a total coloring of the space $X$ in the $R$-extension.

Let $B$ be a basis of the space $X$, and let $B=\bigcup_nB_n$ be a partition of $B$ into countably many subbases. Define the coloring $c$ by the following description. If $x\in\dom(p)$ then $c(x)=p(x)$. If $x\notin\dom(p)$, then let $\ga\in\gw_1$ be the smallest ordinal such that $x\in M_\ga$ holds, and let $c(x)$ be the first (in some fixed enumeration) element of the basis $B_{c_\ga(x)}$ which contains no points in the set $\dom(p)\cup\bigcup_{\gb\in\ga}M_\gb$ which are $\Gamma$-connected to $x$. Such a set exists by Proposition~\ref{goodproposition}. It is important to check the assumptions of that proposition, i.e.\ the set $C_\ga=X\cap(\dom(p)\cup\bigcup_{\gb\in\ga}M_\gb)$ should be $\Gamma$-good. To see this, if $\ga=0$ it follows from the fact that $\dom(p)$ is $\Gamma$-good. If $\ga$ is a successor of some ordinal $\gb$, then it follows from a Mostowski absoluteness argument for the model $M_\gb$. Finally, if the ordinal $\ga$ is limit then $C_\ga$ is the increasing union of $\Gamma$-good sets $C_\gb$ for $\gb\in\ga$ and therefore $\Gamma$-good as well.

Now, work to verify that the map $c$ is a $\Gamma$-coloring. To see that, suppose first that $x\neq y$ are $\Gamma$-connected points in the domain of $c$, and work to show that they receive different colors. Let $\ga_x, \ga_y$ be the smallest ordinals such that $x\in M_{\ga_x}$ and $y\in M_{\ga_y}$. If the two ordinals are distinct (say $\ga_y\in\ga_x$) then the color $c(y)$ is an open set containing $y$, while the color $c(x)$ is an open set containing no elements of the model $M_{\ga_y}$ connected to $x$, in particular $y\notin c(x)$ and $c(x)\neq c(y)$. If the two ordinals $\ga_x, \ga_y$ are equal to some $\ga$ then $c_\ga(x)\neq c_\ga(y)$ and therefore $c(x)\neq c(y)$. As a special case, if $\ga=0$ and one of the points (say $y$) belongs to $\dom(p)$, then either $x\in\dom (p)$ and $c(x)=p(x)\neq p(y)=c(y)$, or $x\notin\dom(p)$ and then $c(x)$ is an open set containing no elements of $\dom(p)$ connected to $x$, in particular $y\notin c(x)$ and $c(x)\neq c(y)$ again.

Finally, it is clear that the collapse of the continuum forces $c\leq p$ by the definition of the map $c$. The proof of the proposition is complete.
\end{proof}

The story is entirely parallel in the case of Ramsey-$\gs$-centered posets, except the eventual conclusion is stronger.

\begin{definition}
\label{ramseycontroldefinition}
Let $P$ be a Suslin poset. Say that $P$ has \emph{Ramsey control} if provably in ZFC, there is a definable map which assigns to each condition $p\in P$ objects $R, \pi, Q, \tau$ so that

\begin{enumerate}
\item $R$ is a forcing and $Q, \tau$ are $R$-names;
\item $\pi\colon R\to\mathrm{Ord}$ is a function such that preimages of singletons are Ramsey-centered;
\item $R\Vdash \langle Q, \tau\rangle$ is a balanced pair in $P$ and $Q\Vdash\tau\leq \check p$.
\end{enumerate}
\end{definition}

\noindent In other words, there may not be a definable balanced pair for $P$, but there is a definable and suitably centered way of forcing such a pair. The following is a complementary concept from descriptive set theory.

\begin{definition}
An analytic graph $\Delta$ on a Polish space $X$ has \emph{countable Borel $\gs$-bounded chromatic number} if there are Borel sets $B_n\subset X$ for $n\in\gw$ such that $\bigcup_nB_n=X$ and $\Delta$ restricted to every finite subset of $B_n$ has chromatic number at most $n+2$.
\end{definition}

\noindent A good example of a Borel graph with uncountable Borel $\gs$-bounded chromatic number is the diagonal Hamming graph $\mathbb{H}_{<\gw}$ of Definition~\ref{hammingdefinition}.  The following is proved in \cite[Theorem 11.6.5]{z:geometric}:

\begin{fact}
\label{rfact}
Suppose that $P$ has Ramsey control.  Let $\Delta$ be an analytic graph on a Polish space with uncountable Borel $\gs$-bounded chromatic number. Let $\kappa$ be an inaccessible cardinal and $W$ be the derived choiceless Solovay model. In the $P$-extension of $W$, $\Delta$ has uncountable chromatic number.
\end{fact}

\begin{proposition}
\label{ramseycontrolprop}
Let $\Gamma$ be a closed Noetherian graph on a $\gs$-compact Polish space $X$. Suppose that there is a number $m$ such that $\Gamma$ contains no cliques of cardinality $m$. Then the coloring poset $P_\Gamma$ has Ramsey control.
\end{proposition}

\noindent The proof is a verbatim repetition of the proof of Proposition~\ref{liminfcontrolprop} with references to Theorem~\ref{liminfcenteredtheorem} and \cite[Proposition 11.7.3]{z:geometric} replaced with Theorem~\ref{ramseycenteredtheorem} and \cite[Proposition 11.6.3]{z:geometric}. It should be noted that while the Suslin-$\gs$-Ramsey-centeredness of the control poset in this case does not need the Noetherian assumption, the proof of Proposition~\ref{ramseycontrolprop} does need it at the place where the transfinite sequence of colorings is stitched into a single one.

Finally, it is possible to state the proofs of the main theorems of the introduction in their entirety. For Theorem~\ref{maintheorem}(1), let $\Gamma$ be a closed Noetherian graph on a $\gs$-compact Polish space. Let $\kappa$ be an inaccessible cardinal, let $W$ be the choiceless Solovay model associated with $\kappa$, and let $W[G]$ be a generic extension of $W$ obtained by the coloring poset $P_\Gamma$ of Definition~\ref{coloringposetdefinition}. The poset is balanced and it has liminf-centered control by Proposition~\ref{liminfcontrolprop}. By Fact~\ref{liminffact}, in the model $W[G]$, the chromatic number of the Hamming graph $\mathbb{H}_\gw$ of infinite breadth is uncountable, and Theorem~\ref{maintheorem}(1) follows. If there is a finite bound on the size of $\Gamma$-cliques, then the coloring poset $P_\Gamma$ has Ramsey control by Proposition~\ref{ramseycontrolprop}. By Fact~\ref{rfact}, in the model $W[G]$, the chromatic number of the diagonal Hamming graph $\mathbb{H}_{<\gw}$ of infinite breadth is uncountable, proving Theorem~\ref{maintheorem}(2).

\bibliographystyle{plain}
\bibliography{odkazy,shelah, zapletal}

\end{document}